\documentclass[a4paper,10pt]{elsarticle}

\usepackage[utf8x]{inputenc}
\usepackage{amsmath}
\usepackage{amsfonts}
\usepackage{amssymb}
\usepackage{graphicx}
\usepackage{subcaption}
\usepackage{hyperref}
\usepackage{color}
\usepackage[noabbrev,capitalize]{cleveref}
\begin{document}
 \title{Connectivity of orientations of 3-edge-connected graphs}
\author{Florian H\"orsch}
 \ead{Florian.Hoersch@grenoble-inp.fr}
\author{Zolt\'an Szigeti}
\ead{ Zoltan.Szigeti@grenoble-inp.fr}
\address{Univ.~Grenoble~Alpes, Grenoble INP, CNRS, G-SCOP, 46 Avenue F\'elix Viallet, Grenoble, France, 38000.  }
\date{\today}

\newtheorem{Statement}{Statement}
\newtheorem{rem}{Remark}
\newtheorem{prop}{Proposition}
\newtheorem{coro}{Corollary}
\newtheorem{theo}{Theorem}
\newtheorem{conj}{Conjecture}
\newtheorem{claim}{Claim}
\newtheorem{lemma}{Lemma}
\newtheorem{prob}{Problem}
\newtheorem{opprob}{Open Problem}

\newcommand{\cev}[1]{\reflectbox{\ensuremath{\vec{\reflectbox{\ensuremath{#1}}}}}}

\newenvironment{proof}{\noindent \textbf{Proof}}{\rule{2mm}{2mm}}

 %%%%%%%%%%%%%%%%%%%%%%%%%%%%%%%%%%%%%%%%%%%%%%%%%%%%%%
 \begin{abstract}
We attempt to generalize a theorem of Nash-Williams stating that a graph has a $k$-arc-connected orientation if and only if it is $2k$-edge-connected. In a strongly connected digraph we call an arc {\it deletable} if its deletion leaves a strongly connected digraph. Given a $3$-edge-connected graph $G$, we define its Frank number $f(G)$ to be the minimum number $k$ such that there exist $k$ orientations of $G$ with the property that every edge becomes a deletable arc in at least one of these orientations. We are interested in finding a good upper bound for the Frank number. We prove that $f(G)\leq 7$ for every $3$-edge-connected graph. On the other hand, we show that a Frank number of $3$ is attained by the Petersen graph. Further, we prove better upper bounds for more restricted classes of graphs and establish a connection to the Berge-Fulkerson conjecture. We also show that deciding whether all edges of a given subset can become deletable in one orientation is NP-complete.
\end{abstract}
\maketitle

\section{Introduction}

This paper deals with ways of orienting undirected graphs so that the obtained directed graph has certain connectivity properties. Our goal is to generalize classical results of Robbins and Nash-Williams.

Let $G=(V,E)$ be an undirected graph.  For some $F\subseteq E$, {\boldmath $G(F)$} $=(V,F)$ denotes the subgraph induced by $F$. For  a set $X\subseteq V$, the subgraph induced by $X$ is denoted by {\boldmath $G[X]$}. We use {\boldmath $\delta_G(X)$} to denote the set of edges between $X$ and $V- X$ and {\boldmath $d_G(X)$} for $|\delta_G(X)|$.
For some vertex $v \in V$, we call $d(\{v\})$ the {\it degree} of $v$. The graph $G$ is called {\it cubic} if  $d_G(\{v\})=3$ for all $v\in V.$  We say that $G$ is {\it $k$-edge-connected} if $d_G(X)\geq k$ for all nonempty, proper subset $X$ of $V.$ We call $G$ {\it Eulerian} if every vertex of $G$ is of even degree.  For some $e \in E$, we denote by {\boldmath $G/e$} the graph obtained from $G$ by contracting $e$, that is deleting $e$ and identifying its two endvertices. For some $F=\{e_1,\ldots,e_t\}\subseteq E$, we denote $G/e_1/\ldots/e_t$ by {\boldmath $G/F$}. For some subgraph $H$ of $G$, we abbreviate  $G/E(H)$ to {\boldmath $G/H$}. An \textit{orientation} of $G$ is a directed graph $D=(V,A)$ such that each edge $uv\in E$ is replaced by exactly one of the arcs $uv$ or $vu$. Given some $e \in E(G)$, we use $\vec{e}$ to denote the associated arc in the orientation. We say that $G$ or a subset of $V$ is {\it trivial} if it contains only one vertex. We call $G$ {\it essentially (k+1)-edge-connected} if $G$ is $k$-edge-connected and for all edge-cuts of size $k$ one side is trivial. A {\it  cycle} is a connected graph each vertex of which is of degree $2$. A {\it path} is a connected graph in which two vertices are of degree $1$ and all other vertices are of degree $2$. A {\it  cycle packing} is a collection of vertex-disjoint cycles of $G.$ We say that a vertex or an edge is in the cycle packing if it is contained in one of the cycles of the packing.

An edge set $M$ of $G$ is called a {\it matching} if each vertex of $G$ is incident to at most one edge of $M.$ A matching $M$ is {\it perfect} if each vertex of $G$ is incident to exactly one edge of $M.$ We say that $G$ is {\it $k$-edge-colorable} if the edge set of $G$ can be partitioned into $k$ matchings.
\medskip

Let $D=(V,A)$ be a directed graph.  For some $F\subseteq A$, let {\boldmath $D(F)$} $=(V,F)$  denote the subgraph induced on $F$. The subgraph induced by some $X \subseteq V$ is denoted by {\boldmath $D[X]$}.  We use {\boldmath $\delta_D^+(X)$} to denote the set of  arcs from $X$ to  $V-X$ and {\boldmath $\delta_D^-(X)$} for $\delta_D^+(V-X)$.
For a vertex $v \in V$, we call $|\delta_D^+(\{v\})|$ the {\it out-degree} and $|\delta_D^-(\{v\})|$ the {\it in-degree} of $v$. The graph that is obtained from $D$ by replacing each arc by an edge between the same two vertices is called the {\it underlying graph} of $D$. We call $D$ {\it weakly connected} if its underlying graph is connected. We call $D$ {\it strongly connected} if $|\delta_D^+(X)|\geq 1$ for every nonempty, proper subset $X$ of $V$. More generally, we say that $D$ is {\it $k$-arc-connected} if $|\delta_D^+(X)|\geq k$ for every nonempty, proper subset $X$ of $V$. We call $D$ {\it Eulerian} if $|\delta_D^+(\{v\})|=|\delta_D^-(\{v\})|$ for every $v \in V$.
% It is well-known that if $D$ is Eulerian then the number of entering arcs is the same as the number of leaving arcs for every nonempty, proper subset of $V$. Therefore, every Eulerian orientation of a $2k$-edge-connected Eulerian graph results in a directed graph that is $k$-arc-connected. 
For some $a \in A$, we denote by {\boldmath $D/a$} the directed graph obtained from $D$ by contracting $a$, that is deleting $a$ and identifying its head and its tail. For some $F=\{a_1,\ldots,a_t\}\subseteq A$, we denote $D/a_1/\ldots/a_t$ by {\boldmath $D/F$}. For some subgraph $H$ of $D$, we abbreviate  $D/A(H)$ to {\boldmath $D/H$}. Let {\boldmath $\cev{D\phantom{'}}$} denote the orientation that arises from $D$ by reversing the orientation of all arcs. A  {\it circuit} is a strongly connected orientation of a cycle. A {\it directed path} is an orientation of a path such that at most one arc enters and at most one arc leaves each vertex. Subscripts may be omitted when the graph or directed graph is clear from the context. We also use basic notions of complexity theory which can be found in Chapter 15 of \cite{koretekgrk}.

\medskip
As one of the first important results in the theory of graph orientations, Robbins proved in 1939 that a graph has a strongly connected orientation if and only if it is 2-edge-connected \cite{robb}. This was later generalized by Nash-Williams \cite{N60}
 who proved that for any positive integer $k$, a graph has a $k$-arc-connected orientation if and only if it is $2k$-edge-connected. This naturally raises the question whether odd edge-connectivity also yields distinctive orientability properties. Our approach to this consists in relaxing the goal to obtain exactly one orientation of the graph to allowing several of them. We say that an arc is {\it deletable} in a $k$-arc-connected orientation of a $(2k+1)$-edge-connected graph if its deletion leaves it $k$-arc-connected. We ask how many orientations are necessary for each edge of the original graph to become a deletable arc in at least one of the orientations. Surprisingly, the number of necessary orientations is bounded by a constant depending only upon $k$. This is a consequence of a theorem of DeVos, Johnson and Seymour \cite{djs}. We focus on  the case $k=1$, meaning we want to find orientations of a 3-edge-connected graph such that for every edge of the graph, the deletion of the associated arc leaves a strongly connected graph in at least one of the orientations. In honor of Andr\'{a}s Frank who proposed this problem and had an immense impact on the development of the theory of graph orientations, we call the minimum number of necessary orientations for a graph $G$ its {\it Frank number} {\boldmath $f(G)$}. Observe that the Frank number of any $4$-edge-connected graph is $1$ as it has a $2$-arc-connected orientation by the theorem of Nash-Williams. On the other hand, any graph $G$ containing a $3$-edge-cut has Frank number at least 2. This follows directly from the fact that in any strongly connected orientation of $G$, there is one arc of the $3$-edge-cut that is oriented differently than the other two arcs. This arc cannot be deletable in this orientation, so at least one more orientation is needed. It is an interesting question to find upper bounds for the Frank number of graphs. A first constant bound can easily be obtained by the following theorem of DeVos, Johnson and Seymour \cite{djs}:

\begin{theo}
Let $G=(V,E)$ be a 3-edge-connected graph. Then there is a partition $\{E_1,\ldots,E_9\}$ of $E$ such that $G-E_i$ is $2$-edge-connected for all $i=1,\ldots,9$.
\end{theo}

This implies the following:

\begin{coro}
Every $3$-edge-connected graph $G$ satisfies $f(G)\leq 9$.
\end{coro}

Indeed, by Robbins' Theorem,  for all $i=1,\ldots,9$, there is a strongly connected orientation of $G-E_i$. Giving an arbitrary orientation to the edges of $E_i$ yields an orientation in which the arcs of $\vec{E}_i$ are deletable.
\medskip

The main contribution of this paper is to further narrow down the values attained by the Frank number. We first show a better upper bound. 

  \begin{theo}\label{7}
Every $3$-edge-connected graph $G$ satisfies $f(G)\leq 7$.
\end{theo}

In attempt to improve on this, we also establish a relationship between our problem and a well-known conjecture about matchings in cubic graphs, the conjecture of Berge-Fulkerson mentioned in Section \ref{prem}.

 \begin{theo}\label{bftheo}
Every $3$-edge-connected graph $G$ satisfies $f(G)\leq 5$ unless the conjecture of Berge-Fulkerson fails.
\end{theo}

Further, we prove a stronger bound for two more restricted classes of $3$-edge-connected graphs.

 \begin{theo}\label{3col}
Every $3$-edge-connected $3$-edge-colorable graph $G$ satisfies $f(G)\leq 3$.
\end{theo}

 \begin{theo}\label{34ec}
Every essentially $4$-edge-connected graph $G$satisfies $f(G)\leq 3$.
\end{theo}

For the lower bound, we show that there are graphs whose Frank number is strictly bigger than 2, more precisely:

  \begin{theo}\label{pet3}
The Frank number of the Petersen graph is 3.
\end{theo}

A drawing of the Petersen graph can be found in Figure \ref{fig0}.

\begin{figure}[h]
    \centering
        \includegraphics[width=.2\textwidth]{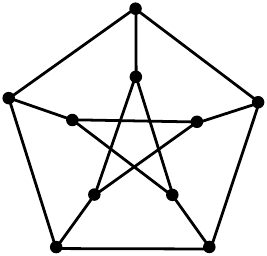}
        \caption{The Petersen graph}\label{fig0}
\end{figure}
\medskip

Given a directed graph $D$, we call a set $F\subseteq A(D)$ {\it deletable} if  $D-f$ is strongly connected for all $f \in F$. 
Given a graph $G$, we call a set $F\subseteq E(G)$ {\it deletable} if there exists an orientation $\vec G$ of $G$ such that $\vec F$ is  deletable in $\vec G$.
\medskip

One of the main difficulties in improving the upper bound on the Frank number consists in finding a useful class of deletable sets. We consider the problem of testing algorithmically whether a set is deletable. 
More formally, we define the following problem:
\medskip

\noindent{\bf DELETABILITY}
\smallskip

\noindent Instance: A graph $G=(V,E)$ and a set $S \subseteq E$.
\smallskip

\noindent Question: Is there an orientation $D$ of $G$ such that $D-\vec{s}$ is strongly connected for all $s \in S$?
\bigskip

The following result shows that an efficient algorithm for DELETABILITY seems out of reach. This implies that a good characterization of deletable sets is hard to obtain.
\begin{theo}\label{nphard}
DELETABILITY is NP-complete for cubic 3-edge-connected graphs.
\end{theo}
In Section \ref{prem}, we present several classical results we will make use of and make some preparatory observations. Also, we introduce an auxiliary graph that will help to reduce the problems to cubic graphs later. In Section \ref{sec3ec}, we deal with the general case of $3$-edge-connected graphs proving Theorems \ref{7} and \ref{bftheo}. Section \ref{e4ecg} is concerned with essentially $4$-edge-connected graphs, in particular the proof of Theorem \ref{34ec}. In Section \ref{petorient}, we prove Theorem \ref{pet3}.  Theorem \ref{nphard} is proven in Section \ref{algo}. Finally, in Section \ref{conc} we conclude our work and give directions for further research on this topic.  

\section{Preliminaries}\label{prem}

In the first part of this section, we give classical results we will make use of later. In the second part, we make some easy preparatory observations which will prove useful later. In the third part, we introduce a way to construct a cubic graph from an arbitrary graph of minimum degree at least $3$ and give some basic properties of it.

\subsection{Previous results}
\medskip

The following result was proven by Nash-Williams \cite{N60} and is the starting point of our work as it characterizes the graphs admitting a $k$-arc-connected orientation.

\begin{theo}\label{faible}
A graph has a $k$-arc-connected orientation if and only if it is $2k$-edge-connected.
\end{theo}

In fact, there is an immense strengthening of this theorem. For any two vertices $u,v$ of a graph $G$, let {\boldmath $\lambda_G(u,v)$} be the maximum number of edge-disjoint paths between $u$ and $v$. An orientation $D$ of $G$ is called {\it well-balanced} if for any $u,v \in V(G)$, there exist at least $\lfloor \frac{\lambda_G(u,v)}{2}\rfloor$ directed paths from $u$ to $v$ and also from $v$ to $u$ in $D$. An {\it odd-vertex pairing} of $G$ is a perfect matching in the complete graph whose vertex set is the set of vertices of odd degree of $G$. An odd-vertex pairing $P$ is called {\it admissible} if the restriction of any Eulerian orientation of $G+P$ to $E(G)$ yields a well-balanced orientation of $G$. Nash-Williams \cite{N60} proved the following:

\begin{theo}\label{fort}
Every graph has an admissible odd-vertex pairing.
\end{theo}

Observe that this implies that every graph has a well-balanced orientation. Also, as well-balanced orientations of $2k$-edge-connected graphs are $k$-arc-connected, Theorem \ref{fort} implies Theorem \ref{faible}.
\medskip

There are several other well-known results we make use of in this article. The first one is about packing spanning trees and is also due to Nash-Williams \cite{NashWilliams1961}.

\begin{theo}\label{rbrcvrnt}
Every $2k$-edge-connected graph has $k$ edge-disjoint spanning trees. 
\end{theo}

The next one concerns matching theory and was proven by Petersen \cite{petersen}.

\begin{theo}\label{pettheo}
Every cubic $2$-edge-connected graph has a perfect matching. 
\end{theo}

The next one concerns minimum edge cuts and can for example be found as Theorem 7.1.2 in \cite{book}.

\begin{prop}\label{crossfree}
Let $G$ be a $3$-edge-connected graph and let $\delta(X)$ and $\delta(X')$ be $3$-edge-cuts of $G$. Then $\delta(X)$ and $\delta(X')$ are not crossing, i.e. one of $X-X'$, $X'-X$, $X \cap X'$  or $V(G)-(X \cup X') $ is empty.
\end{prop}

Let $G=(V,E)$ be a graph and $T\subseteq V$. A {\it $T$-join} is defined to be a set $F \subseteq E$ such that the set of odd degree vertices of $G(F)$ is $T$. We use the following characterization of the existence of $T$-joins that can be found as Proposition 12.7 in \cite{koretekgrk}.

\begin{prop}\label{tjn}
Let $G=(V,E)$ be a graph and $T \subseteq V$. Then $G$ contains a $T$-join if and only if every connected component of $G$ contains an even number of elements of $T$. 
\end{prop}

The following theorem is due to Menger \cite{men} and is a fundamental characterization of $k$-edge-connected graphs.

\begin{theo}\label{mentheo}
A graph $G=(V,E)$ is $k$-edge-connected if and only if $\lambda_G(u,v)\geq k$ for all $u,v \in V$.
\end{theo}

Further, we mention an intensively studied conjecture which was proposed independently by Berge and Fulkerson \cite{bergefulk}.

\begin{conj}\label{bf}
Every cubic $2$-edge-connected graph has a set of six perfect matchings such that every edge is contained in exactly two of them.
\end{conj}

We also consider the following algorithmic problem which is well-known in the literature:

\bigskip
\noindent {\bf Monotone Not-all-equal-3SAT(MNAE3SAT)}

\smallskip
\noindent Instance: A set $X$ of boolean variables, a formula consisting of a set $\mathcal{C}$ of clauses each containing 3 distinct variables, none of which are negated.

\smallskip
\noindent Question: Is there a truth assignment to the variables of $X$ such that every clause in $\mathcal{C}$ contains at least one true and at least one false literal?
\bigskip

This problem will be used in the reduction in Section \ref{algo} which is justified by the following result due to Schaefer \cite{schaefer}.
\begin{theo}\label{sathard}
MNAE3SAT is NP-complete.
\end{theo}

\subsection {Preparatory results}
\medskip

The following two results show that connectivity properties are maintained when contracting or blowing up sufficiently connected subgraphs. As they are of basic nature, they are given without proof.

  \begin{prop}\label{dgcnnctd}
  Let $G$ be a graph.
  
(a) If $G$ is $k$-edge-connected, then so is any contraction of $G.$

(b) If $G$ is essentially $k$-edge-connected, then so is any contraction of $G.$
  \end{prop}

 \begin{prop}\label{strnglcnnctd}
 For a subgraph $Q$ of a directed graph $D,$ 
 
 (a) if $D$ is strongly connected, then so is  $D/Q$,
 
 (b) if $D/Q$ and $Q$ are strongly connected, then so is $D.$
 \end{prop}

 The following observation concerns Eulerian orientations and follows easily from Proposition 5 of \cite{kiralyszigeti}. For the sake of completeness we provide an easy proof for it.

\begin{prop}\label{eulspecial}
Let $G=(V,E)$ be an Eulerian graph and  $\{e_v,f_v\}$ two edges incident to $v$ for all $v\in V'\subseteq V$. Then there is an Eulerian orientation of $G$ such that exactly one of $e_v$ and $f_v$ enters $v$ for  all $v\in V'$.
\end{prop}

\begin{proof}
Let $G'$ be the graph obtained from $G$ by detaching each vertex $v\in V'$ into two vertices $u_v$ and $w_v$ such that $u_v$ is incident to $ \{e_v,f_v\}$ and $w_v$ is incident to $\delta_G(\{v\})-\{e_v,f_v\}$  in $G'$. As $G$ is Eulerian, so is $G'.$ Hence there exists an Eulerian orientation $D'$ of $G'$. By identifying $u_v$ and $w_v$ in $D'$ for all $v\in V'$, we obtain  the required orientation.
\end{proof}

\bigskip
%We will next introduce a little more notation related to the Frank number.
%Given a directed graph $D$, we call a set $F\subseteq A(D)$ {\it deletable} if  $D-f$ is strongly connected for all $f \in F$. 
%Given a graph $G$, we call a set $F\subseteq E(G)$ {\it deletable} if there exists an orientation $\vec G$ of $G$ such that $\vec F$ is  deletable in $\vec G$.

The following result is a direct consequence of the definition of strongly connected directed graphs.

\begin{prop}\label{efface}
Given a directed graph $D=(V,A)$, a set $F\subseteq A$ is deletable if and only if $\delta_D^-(X)$ contains either at least one arc of $A-F$ or at least two arcs for every nonempty, proper subset $X$ of $V$.    
\end{prop}

Consequently, given a graph $G$, a subset $F \subseteq E(G)$ is deletable if and only if there is an orientation satisfying the above properties.
\medskip

Finally, we show one more result about strongly connected orientations of $3$-edge-connected graphs which we need in the proof of Theorem \ref{34ec}.

 \begin{lemma}\label{cyclearc}
 Let $D$ be a strongly connected orientation of a $3$-edge-connected graph $G$ and $C$ a circuit of $D$. 
 %such that no arc of $C$ belongs to a cut of size two in $G.$ 
 Then $C$ contains an arc $a$ such that $D-a$ is strongly connected.
 \end{lemma}

 \begin{proof} Let $(G, D, C)$ be a counterexample that minimizes the number of vertices of $D.$ Let $e$ be an edge of $G$ that is incident to a vertex of $C$ and that does not belong to $C.$ By the 3-edge-connectivity of $G$, $e$ exists. Since $D$ is strongly connected, $\vec e$ belongs to a directed path $P$ whose end-vertices belong to $C$ but whose internal vertices do not. Then $P$ can be extended by a possibly trivial directed subpath of $C$ to form a circuit $C^*$. Let $(G', D', C')$ be obtained from $(G, D,C)$ by contracting $C^*.$ Then, by Propositions \ref{dgcnnctd}(a) and \ref{strnglcnnctd}(a), the assumptions of the lemma are satisfied for $(G', D',C')$. By the minimality of $(G, D, C)$, $C'$ contains an arc $a'$  such that $D'-a'$ is strongly connected. Let $a$ be the arc of $C$ in $D$ that corresponds to $a'$. Since $D'-a'$ and $C^*$ are strongly connected, by Proposition \ref{strnglcnnctd}(b), so is $D-a$.   
 \end{proof}

 \subsection{Cubic extensions}\label{reduc}
 \medskip

We introduce for any graph $G=(V,E)$ of minimum degree at least 3 an auxiliary graph {\boldmath $H_G$} that is cubic. For each vertex $v \in V$ of degree at least 4, $H_G$ contains a set $S_v$ of $d_G(\{v\})$ vertices. For each vertex of degree 3, let $S_v=\{v\}$. Next,  for each $v \in V$ of degree at least 4, we add a cycle $C_v$ whose vertex set is $S_v$. Finally, for each edge $uv \in E$, we add an edge between $S_u$ and $S_v$ to $H_G$. We do this in a way so that $H_G$ becomes cubic. We call $H_G$ a {\it cubic extension} of $G.$ Note that $H_G$ is not unique. This ambiguity has no consequences though.

\begin{prop}\label{red34} Let $G=(V,E)$ be a graph of minimum degree at least $3$ and $H_G$ be a cubic extension of $G$.
	\begin{itemize}
		\item [(a)] If $G$ is 3-edge-connected and $G-v$ is connected for all $v\in V$, then  $H_G$ is 3-edge-connected.
		\item [(b)] If $G$ is essentially 4-edge-connected and $G-v$ is  2-edge-connected for all $v\in V$, then $H_G$ is essentially 4-edge-connected.
	\end{itemize}
\end{prop}

\begin{proof}
(a) Assume for a contradiction that $d_{H_G}(X)\leq 2$ for some nonempty, proper subset $X$ of $V(H_G).$ Since $G$ is 3-edge-connected, there is at least one $v \in V$ such that $S_v\cap X$ and $S_v-X$ are nonempty. It follows that $2\leq d_{C_v}(X)\leq d_{H_G}(X)\leq 2$. This yields that for every $u \in V-v$ we have $S_u \subseteq X$ or $S_u \subseteq V(H_G)-X$ and for all $uw \in E$ with $v \notin \{u,w\}$ we have $S_u \cup S_w \subseteq X$ or $S_u \cup S_w \subseteq V(H_G)-X$. If there are vertices $u,w \in V$ such that $S_u \subseteq X$ and $S_w \subseteq V(H_G)-X$, it follows that $G-v$ is not connected, contradicting the assumption. Therefore, by symmetry we may assume  that $X$ is a nonempty, proper subset of $S_v$. We then have $d_{C_v}(X) \geq 2$ and there is at least one additional edge between $X$ and $V(H_G)-S_v$, a contradiction to  $d_{H_G}(X)\leq 2$.

(b) By (a), $H_G$ is 3-edge-connected. For the sake of a contradiction, suppose that there is some non-trivial, proper subset $X$ of $V(H_G)$ such that $d_{H_G}(X)=3$. If there are two vertices $u,v \in V$ such that $S_u\cap X$, $S_u-X$, $S_v\cap X$ and $S_v-X$ are nonempty, we have $2+2\leq d_{C_u}(X)+d_{C_v}(X)\leq d_{H_G}(X)\leq 3$, a contradiction.

 Now consider the case that there is exactly one $v \in V$ such that $S_v\cap X$ and $S_v-X$ are nonempty. We have that $d_{H_G}(X)-d_{C_v}(X)\leq 1$. It follows that in $H_G$ there is at most one edge between $X-S_v$ and $V(H_G)-X-S_v$. If  $X-S_v$ and $V(H_G)-X-S_v$ are nonempty, then $G-v$ is not $2$-edge-connected, a contradiction to the assumption. By symmetry, we may therefore assume that $X\subset S_v$. We have that $d_{C_v}(X)=2$ and there are $|X|$ edges between $X$ and $V(H_G)-S_v$. It follows that $|X|=1$, which is a contradiction.

Finally assume that $S_v\subseteq X$ or $S_v\cap X=\emptyset$ for all $v \in V$. Let $X'=\{v \in V:S_v \subseteq X\}$. As $G$ is essentially $4$-edge-connected, we may assume by symmetry that $X'=\{v\}$ for some vertex $v$ of degree $3$. This yields that $|X|=|S_v|=1$, which is a contradiction.
\end{proof}

 \section{3-edge-connected graphs}\label{sec3ec}
This section is dedicated to proving Theorems \ref{7}, \ref{bftheo} and \ref{3col}. In the first part, we show that a certain class of edge sets is deletable. After, we show how to cover cubic $3$-edge-connected graphs with such sets. Next, we use this to conclude cubic versions of Theorems \ref{7} and \ref{bftheo} and to prove Theorem \ref{3col}. Finally, we extend this to obtain the general versions of Theorems \ref{7} and \ref{bftheo}.

\subsection{A class of deletable edge sets}
\medskip

Given a packing $\mathcal{C}$ of cycles in a 3-edge-connected graph $G$, the {\it special set} of $\mathcal{C}$ is defined to be the set of edges in $E(G)- E(\mathcal{C})$ that belong to no 3-edge-cut of $G/\mathcal{C}$.

\begin{lemma}\label{cubicmatching}
Let $M$ be the special set of  a cycle packing ${\cal C}$ of a $3$-edge-connected graph $G$. Then $M$ is deletable. 
\end{lemma}

\begin{proof}
Let $G'=G/\mathcal{C}$. Since $G$ is 3-edge-connected, so is $G'$ by Proposition \ref{dgcnnctd}(a).
Consider a well-balanced orientation $D'$ of $G'$ which exists by Theorem \ref{fort}. Then $D'$ is strongly connected. Let $D$ be the  orientation of $G$ obtained from $D'$ by orienting all cycles of $\mathcal{C}$ as a circuit. 

We have to show that $D-\vec f$ is strongly connected for all $f \in M.$  By Proposition \ref{strnglcnnctd}(b), it is enough to show that $D'-\vec f$ is strongly connected for all $f\in M.$
 Let $\vec f=uv$ for some $f \in M$ and suppose that there exists some non-empty, proper subset $X$ of  $V(D')$ with $|\delta^+_{D'-\vec f}(X)|=0$. Obviously $u \in X$ and $v \in V(D')-X$. Since $G'$ is 3-edge-connected and $f$ belongs to no 3-edge-cut in $G'$, Theorem \ref{mentheo} guarantees that $\lambda_{G'}(u,v) \geq 4$. As $D'$ is well-balanced, it follows that $0=|\delta^+_{D'-\vec f}(X)|=|\delta^+_{D'}(X)|-1\geq \lfloor\frac{\lambda_{G'}(u,v)}{2}\rfloor-1\geq 2-1=1$, a contradiction.
\end{proof}

%First, we prove the result for cubic graphs and then we show how it implies the general case.
 %%%%%%%%%%%%%%%%%%%%%%%%%%%%%%%%%%%%%%%%%%%%%%%%%%%%%%
\subsection{Covering cubic graphs with special sets}
\medskip

 %%%%%%%%%%%%%%%%%%%%%%%%%%%%%%%%%%%%%%%%%%%%%%%%%%%%%%

In the following we show that any cubic 3-edge-connected graph can be covered by 7 special sets. 
For technical reasons, we will need the following slight strengthening.

\begin{lemma}\label{strong7packing}
For every cubic 3-edge-connected graph, there exist 7 cycle packings satisfying the following conditions:
\begin{itemize}
    \item [(a)] Every edge is in the special set of at least one cycle packing.
    \item [(b)] Every edge is in exactly 4 of the cycle packings.
\end{itemize}
\end{lemma}

\begin{proof}
For the sake of a contradiction, let $G=(V,E)$ be a counterexample to the lemma that minimizes $|V|$. 
%The following result will allow us to obtain more structural properties we can benefit from.
\begin{claim}\label{oubsfjd}
$G$ is essentially 4-edge-connected.
\end{claim}
\begin{proof}
For the sake of a contradiction, let $\{A_1,A_2\}$ be a partition of $V(G)$ such that  $|A_i|\geq 2$ and   a 3-edge-cut  $F:=\{e_1,e_2,e_3\}$ exists between $A_1$ and $A_2$. Construct the graphs $G_i$ from $G$ by contracting $A_{3-i}$  to $v_i$. As $G_i$ is cubic, 3-edge-connected by Proposition \ref{dgcnnctd}(a) and smaller than $G$, there exists a set of cycle packings $\mathbb{C}^i=\{\mathcal{C}_{1}^i,\ldots,\mathcal{C}_{7}^i\}$ of $G_i$ satisfying $(a)$ and $(b)$. 

Observe that since $G_i$ is cubic, $(b)$ implies that for $j\in \{1,2,3\}$, there are exactly two cycle packings in $\mathbb{C}^i$ that contain $\{e_1,e_2,e_3\}-\{e_j\}$. It follows that $v_i$ is in exactly 6  cycle packings of $\mathbb{C}^i.$ By relabeling if needed, we may assume that $\mathcal{C}_{1}^i$ is the cycle packing that does not contain $v_i$ and  $\{e_1,e_2,e_3\}-\{e_j\}$ is contained in $\mathcal{C}_{2j}^i$ and $\mathcal{C}_{2j+1}^i$. We may also assume, by (a), that $e_j$ is in the special set of $\mathcal{C}_{2j}^i$. 

%Observe that $(b)$ implies that every vertex is in exactly 6 of the 7 cycle packings. By relabeling if needed, we may assume that $\mathcal{C}_{1}^i$ is the cycle packing that does not contain $v_i$. Also, for $j\in \{1,2,3\}$, there are exactly two cycle packings in $\mathbb{C}^i$ that contain $\{e_1,e_2,e_3\}-\{e_j\}$. We may assume that these are $\mathcal{C}_{2j}^i$ and $\mathcal{C}_{2j+1}^i$. By (a), we may assume that $e_j$ is in the special set of $\mathcal{C}_{2j}^i$. 

We  construct $\mathbb{C}=\{\mathcal{C}_1,\ldots,\mathcal{C}_7\}$ so that $E(\mathcal{C}_k)=E(\mathcal{C}_{k}^1)\cup E(\mathcal{C}_{k}^2)$ for $k=1,\ldots,7$. Observe that this is a set of seven cycle packings. We  finish the proof by showing that $\mathbb{C}$ satisfies $(a)$ and $(b)$.

First observe that $(b)$ follows directly from the construction and the fact that an edge is in $\mathcal{C}_k$ if and only if it is in $\mathcal{C}_k^1$ or $\mathcal{C}_k^2$. For $(a)$, let first $e$ be an edge in $G[A_i]$. By $(a)$, there exists a $k \in \{1,\ldots,7\}$ such that $e$ is in the special set of $\mathcal{C}_{k}^i$. First observe that $e$ is  in $E(G)- E(\mathcal{C}_k)$. If $e$ is in a 3-edge-cut $F'$ of $G/\mathcal{C}_k$, since $F'$ is not a 3-edge-cut of $G_i/\mathcal{C}_{k}^i$, $F'$ contains an edge of $G[A_{3-i}]$. This yields that $F$ and $F'$ are crossing 3-edge-cuts of $G$, a contradiction to Proposition \ref{crossfree}.

Now consider the edge $e_j$ for some $j \in \{1,2,3\}.$ As $e_j\in E(G_i)- E(\mathcal{C}_{2j}^i)$, we have $e_j\in E(G)- E(\mathcal{C}_{2j}).$
Again, assume that $e_j$ is in a 3-edge-cut $F'$ of $G/\mathcal{C}_{2j}$. As $F'$ is not a $3$-edge-cut in $G_1/\mathcal{C}_{2j}^1$ and $G_2/\mathcal{C}_{2j}^2$, we obtain that $F'$ and $F$ are crossing in $G$ contradicting Proposition \ref{crossfree}. This finishes the proof of the claim. 
\end{proof} 
\bigskip

By Theorem \ref{pettheo}, $G$ contains a perfect matching {\boldmath $M$}. Since $G$ is cubic, the connected components of $G-M$ form a cycle packing  {\boldmath $\mathcal{C}_1$}. Now consider the graph {\boldmath $G'$} $:=G/\mathcal{C}_1$ (including arising loops) and let {\boldmath $T$} be its set of odd-degree vertices.

\begin{claim}\label{3tjoins}
The edge set of $G'$ can be partitioned into three $T$-joins {\boldmath $F_1$}, {\boldmath $F_2$} and {\boldmath $F_3$}.
\end{claim}

\begin{proof}
As $G$ is essentially 4-edge-connected by Claim \ref{oubsfjd}, $G'$ is essentially 4-edge-connected by Proposition \ref{dgcnnctd}(b). Every vertex $v$ of $G'$ corresponds to a cycle $C \in  \mathcal{C}_1$. It follows that $d_{G'}(v)\geq d_G(C)\geq 4$, so $G'$ is $4$-edge-connected. By Theorem \ref{rbrcvrnt}, there exist two edge-disjoint spanning trees $F'_1,F'_2$ of $G'$. By Proposition \ref{tjn}, each of them  contains a $T$-join $F_i,i=1,2$. As $F_1 \cup F_2$ is Eulerian, $F_3=E(G')-F_1-F_2$ is also a $T$-join.
\end{proof}

\begin{claim}\label{2Vjoins}
For $i=1,2,3$, there exist $V$-joins $S_{2i}$ and $S_{2i+1}$ of $G$ such that $S_{2i}\cap S_{2i+1}=F_i$ and $S_{2i}\cup S_{2i+1}=(E-M)\cup F_i.$
\end{claim}

\begin{proof} For $i=1,2,3$, let {\boldmath $T_{i}$} be the set of vertices in $V$ not incident to an edge in $F_i$. Let $C \in\mathcal{C}_1$ and let $v_C$ be the associated vertex in $V(G')$. Observe that, as $G$ is cubic and $F_i\subseteq M$ is a matching in $G$, we obtain $|V(C)|\equiv d_G(V(C)) = d_{G'}(v_c)$ and $|V(C)\cap V(F_i)|\equiv d_{F_i}(V(C)) = d_{F_i}(v_c)$. As $F_i$ is a $T$-join in $G'$, this yields $|T_i\cap V(C)|=|V(C)|-|V(C)\cap V(F_i)|\equiv d_{G'}(v_C)-d_{F_i}(v_C)\equiv 0$, so $|T_i\cap V(C)|$ is even. %As the vertices of $T$ correspond to odd cycles of $\mathcal{C}_1$ and the vertices in $V(G')-T$ correspond to even cycles, we can extend $F$ to a $V$-join $S$ of $G$ adding only edges of $E-M$. 
%Since $d_G(V(C))$ and $d_{F_i}(V(C))$ are of the same parity, $|T_i\cap V(C)|= d_G(V(C))-2\alpha_C-d_{F_i}(V(C))$ is even.
 Hence, by Proposition \ref{tjn}, we obtain that $G-M$ contains a $T_i$-join {\boldmath $N_i$}. Let {\boldmath $S_{2i}$} $:=F_i\cup N_i$ and {\boldmath $S_{2i+1}$} $:=F_i\cup(E-M- N_i )$. By construction, we have that $S_{2i}$ and $S_{2i+1}$ are $V$-joins in $G$ such that $S_{2i}\cap S_{2i+1}=F_i$ and $S_{2i}\cup S_{2i+1}=(E-M)\cup F_i.$
\end{proof}
\bigskip

For $j=2,\dots,7$, we define {\boldmath $\mathcal{C}_{j}$}  to be the set of nontrivial connected components of $G-S_{j}$. Observe that all of them are cycles as $S_j$ is a $V$-join and $G$ is cubic. 

\begin{claim}\label{igzberqkljds}
 $\mathcal{C}_1,\ldots,\mathcal{C}_7$ satisfy $(a)$ and $(b)$. 
\end{claim}

\begin{proof}
(a) For $e\in M$, since $G$ is essentially 4-edge-connected, $e$ is in the special set of $\mathcal{C}_1$. 

For $e\in E-M$, let $f$ and $g$ be the two edges of $M$ adjacent to $e$. Since $F_1,F_2$ and $F_3$ are disjoint, there is an $F_i$ that contains neither $f$ nor $g$. Then, since $G$ is cubic and by Claim \ref{2Vjoins}, one of the $V$-joins $S_{2i}$ and $S_{2i+1}$, say $S_j$, contains $e$ but none of the edges  adjacent to $e$. It follows that both endvertices of $e$ in $G$ are in cycles of $\mathcal{C}_{j}$. As $G$ is essentially 4-edge-connected, it follows that both endvertices of $e$ in $G/\mathcal{C}_{j}$ are of degree at least $4$.  As $G$ is essentially 4-connected, so is $G/\mathcal{C}_{j}$ by Proposition \ref{dgcnnctd}(b). This yields that $e$ is in no 3-edge-cut of $G/\mathcal{C}_{j}$ and so $e$ is in the special set of $\mathcal{C}_{j}$. 
\medskip

(b) For $e\in M$, by Claim \ref{3tjoins}, $e$ is in exactly one $F_i$, say $F_1.$ Then, by Claim \ref{2Vjoins}, $e$ is in $\mathcal{C}_4,\dots,\mathcal{C}_7$ and not in  $\mathcal{C}_1,\mathcal{C}_2,\mathcal{C}_3$.

For $e\in E-M$, $e$ is in $\mathcal{C}_1$  and, by Claim \ref{2Vjoins},  in exactly one of $\mathcal{C}_{2i}$ and $\mathcal{C}_{2i+1}$ for $i=1,2,3$. %So again, it is in 4 of the $\mathcal{C}_j$ in total. 
\end{proof}
\bigskip

Claim \ref{igzberqkljds} finishes the proof of Lemma \ref{strong7packing}.
\end{proof}

\subsection{Cubic case}
\medskip

We first show how to conclude a cubic version of Theorem \ref{7}.
\begin{theo}
Let $G$ be a cubic $3$-edge-connected graph. Then $f(G)\leq 7$.\label{7cub}
\end{theo}
\begin{proof}
 Lemma \ref{strong7packing} yields that $E(G)$ can be covered by $7$ special sets $S_1,\ldots,S_7$. By Lemma \ref{cubicmatching}, there exist orientations $D_1,\ldots,D_7$ of $G$ such that $S_i$ is deletable in $D_i$ for $i=1,\ldots,7$. It follows that the Frank number of $G$ is at most $7$.
\end{proof}
\medskip

Next, we use Lemma \ref{cubicmatching} to show that perfect matchings with a certain additional property are deletable. As corollaries, we obtain Theorem \ref{3col} and a cubic version of Theorem \ref{bftheo}.
\begin{lemma}\label{perfect}
Let $M$ be a perfect matching of a cubic $3$-edge-connected graph $G$ intersecting every $3$-edge-cut of $G$ in exactly one edge. Then $M$ is deletable.
\end{lemma}
\begin{proof}
As $G$ is cubic and $M$ is a perfect matching of $G$, the connected components of $G-M$ form a packing ${\cal C}$ of cycles. 
We show that $G/{\cal C}$ is 4-edge-connected.
By Proposition \ref{strnglcnnctd}(a) and since $G$ is 3-edge-connected, so is $G/{\cal C}$. A 3-edge-cut of $G/{\cal C}$ would provide a 3-edge-cut of $G$ intersecting $M$ in 3 edges contradicting the assumption. 
It follows that $M$ is the special set of ${\cal C}$ and therefore deletable by Lemma \ref{cubicmatching}.
\end{proof}
 \medskip

We first show how to conclude Theorem \ref{3col} from Lemma \ref{perfect}.
\medskip

\begin{proof} (of Theorem \ref{3col})
Let $G$ be a $3$-edge-colorable $3$-edge-connected graph. Then $G$ is cubic and has $3$ disjoint perfect matchings $M_1,M_2,M_3$ covering the edge set of $G$.
Let $\delta(X)$ be a 3-edge-cut of $G.$ Since $G$ is cubic and $d(X)=3$, we obtain that $|X|$ is odd. Then, since  $M_i$ is a perfect matching, we obtain that $\delta(X)$ intersects each $M_i$. As $d(X)=3$ and the matchings are disjoint, we obtain that $\delta(X)$ intersects each of $M_1, M_2, M_3$ exactly once. It follows by Lemma \ref{perfect} that each of  $M_1, M_2, M_3$ is deletable, so $f(G)\leq 3$.
\end{proof}
\medskip

Next, we prove in a similar way the following cubic version of Theorem \ref{bftheo}. 
\begin{theo}\label{fn5}
Let $G$ be a cubic 3-edge-connected graph that satisfies Conjecture \ref{bf}. Then $f(G) \leq 5$.
\end{theo}

\begin{proof}
By assumption, there exist 6 perfect matchings $M_1, \ldots, M_6$ of $G$ covering each edge of $G$ exactly twice. 

Let $\delta(X)$ be a 3-edge-cut of $G.$ Since $G$ is cubic and $d(X)=3$, we obtain that $|X|$ is odd. Then, since  $M_i$ is a perfect matching, $\delta(X)$ intersects each $M_i$. Since each of the 3 edges of $\delta(X)$  belongs to exactly 2 $M_i$'s, $\delta(X)$ intersects each of $M_1, \ldots, M_6$ exactly once. 
 It follows by Lemma \ref{perfect} that each of  $M_1, \ldots, M_6$ is deletable. As every edge of $G$ is covered by at least one of $M_1,\ldots,M_5$, it follows that $f(G)\leq 5$. 
\end{proof}

 %%%%%%%%%%%%%%%%%%%%%%%%%%%%%%%%%%%%%%%%%%%%%%%%%%%%%%
%We will first prove that all sets of a certain kind are deletable and then that each graph can be covered with such.

% %%%%%%%%%%%%%%%%%%%%%%%%%%%%%%%%%%%%%%%%%%%%%%%%%%%%%%
%\subsection{Reduction to cubic graphs}
%
%Consider the following statement restricting the problem to cubic graphs.
%
%\begin{Statement}\label{cubic}
%The Frank number of any 3-edge-connected, cubic graph is at most $n$.
%\end{Statement}
%
%The following result justifies us restricting ourselves to cubic graphs for the rest of the article.
%
%\begin{theo}
%For every $n \in \mathbb{N}$, Statement \ref{general} and Statement \ref{cubic} are equivalent.
%\end{theo}

\subsection{Non-cubic case}
\medskip

We first show how to prove the general case of Theorem \ref{7}.
\medskip

\begin{proof} (of Theorem \ref{7}) Let  $G$ be a counterexample minimizing $|V(G)|$.

\begin{claim}\label{cutvertex}
$G$ is 2-vertex-connected.
\end{claim}

\begin{proof}
For the sake of a contradiction, assume that $G$ has a cut vertex $v$. So $G$ has two non-trivial subgraphs $G_1$ and $G_2$ such that $G_1=G/G_2$ and $G_2=G/G_1$. As $G$ is 3-edge-connected, so is $G_i$ by Proposition \ref{dgcnnctd}(a). Since $G_i$  is smaller than $G$, $G_i$ has Frank number at most $7$. So there exist $7$ orientations $D_j^i$ of $G_i$ such that for each edge $e$ of $G_i$, one of $D_j^i-\vec e$ is strongly connected. We can now construct the 7 orientations $D_j$ of $G$ by giving each edge in $G_i$ its orientation in $D_j^i$ also in $D_j$. Now consider an edge $e$ of $G_i$ and let $D_j^i-\vec e$ be strongly connected. Since $D_j^i-\vec e=(D_j-\vec e)/D_j^{3-i}$ and $D_j^{3-i}$ are strongly connected,  Proposition \ref{strnglcnnctd}(b) implies that so is $D_j-\vec e$.
 It follows that $G$ has Frank number at most $7$, a contradiction.
\end{proof}
\bigskip

Let $H_G$ be a cubic extension of $G$ as defined in Section \ref{reduc}. By Claim \ref{cutvertex} and Proposition \ref{red34}(a), $H_G$ is  3-edge-connected. Then, by Theorem \ref{7cub}, the Frank number of $H_G$ is at most $7$, that is there exist 7 orientations $D'_i$ of $H_G$ such that for each edge $e$ of $H_G$, one of $D'_i-\vec e$ is strongly connected. Let $D_i$ be the orientation of $G$ obtained from $D'_i$ by contracting the subgraphs $C_v$ for all $v\in V(G).$ For any $e\in E(G)\subset E(H_G)$, one of $D'_i-\vec e$ is strongly connected, therefore, by Proposition \ref{strnglcnnctd}(a), so is $D_i-\vec e$. It follows that the Frank number of $G$ is at most $7$, a contradiction.
\end{proof}
\bigskip

The same reduction  and Theorem \ref{fn5} show Theorem \ref{bftheo}.

 %%%%%%%%%%%%%%%%%%%%%%%%%%%%%%%%%%%%%%%%%%%%%%%%%%%%%%

\section{Essentially 4-edge-connected graphs}\label{e4ecg}

This section is dedicated to proving Theorem \ref{34ec}. Again, first we prove the result for cubic graphs and then we show how it implies the non-cubic case.

 %%%%%%%%%%%%%%%%%%%%%%%%%%%%%%%%%%%%%%%%%%%%%%%%%%%%%%
\subsection{Cubic case}
\medskip

In the case of essentially 4-edge-connected graphs, we can show that every matching is deletable. We prove the following slightly stronger statement.

\begin{lemma}\label{matching}
Let $G$ be an essentially 4-edge-connected graph, $M$ a matching of $G$ and ${\cal C}$ a cycle packing of $G-M$. Then there exists an orientation of $G$ in which $\vec{M}$ is deletable and each cycle of  ${\cal C}$ is oriented as a circuit.
\end{lemma}

\begin{proof}
Let {\boldmath ${\cal F}$} be the set of maximal 2-edge-connected subgraphs of $G-M$.
Let {\boldmath $G'=(V',E' \dot\cup M)$} be the graph obtained from $G$ by contracting each graph of ${\cal F}$. Note that $G'-M$ is a forest. Since $G$ is essentially 4-edge-connected, by Proposition \ref{dgcnnctd}(b), so is $G'$ and every vertex  of degree 3 in $G'$ is an original vertex of $G.$ 
Then,  since $M$ is a matching of $G$, every vertex $v$ of degree 3 in $G'$ is incident to at least 2  edges $e_v^1, e_v^2$ in $E'$. 
%Then,  since $M$ is a matching of $G$, every vertex of $G'$ is either of degree at least 4 or incident to at most one edge of $M$. 

By Theorem \ref{fort}, there exists an admissible pairing $P$ of $G'$. 
As $G'+P$ is Eulerian, Proposition \ref{eulspecial} yields that $G'+P$ has an Eulerian orientation $\vec{G'}+\vec{P}$ such that for each vertex $v$ of degree 3 in $G'$, one of $\vec e_v^1, \vec e_v^2$  enters $v$ and the other one leaves $v.$  
By the definition of admissible pairings, $\vec{G'}$ is a well-balanced orientation of $G'.$ 

For all $F\in {\cal F}$, by Proposition \ref{dgcnnctd}(a), Theorem \ref{faible} and Proposition \ref{strnglcnnctd}(b), there exists a strongly connected orientation $\vec F$ of $F$ such that each cycle of  ${\cal C}$ contained in $F$ is oriented as a circuit. %Note that each cycle of ${\cal C}$ belongs to some $F\in {\cal F}.$ 

Let $\vec{G}$ be the orientation of $G$ obtained by combining $\vec{G'}$ and $\vec F$ for all $F\in {\cal F}$. Proposition \ref{strnglcnnctd}(b) yields that $\vec{G}$ is strongly connected. Since each cycle $C$ of ${\cal C}$ belongs to some $F\in {\cal F},$ $C$ is oriented as a circuit in $\vec{G}.$

We will finish the proof by showing that $\vec{G}-\vec e$ is strongly connected for all $e \in M$. 
%It suffices to prove that for every subset $X$ of $V(D^*)$, either two arcs or an arc of $A(D^*)-M$ enter $X$.
Since $\vec F$ is strongly connected for all $F\in {\cal F}$ and $\bigcup_{F \in {\cal F}}E(F)$ contains no edge in $M,$ it suffices to prove, by Proposition \ref{strnglcnnctd}(b), that $\vec{G'}-\vec e$ is strongly connected for all $e \in M$.
Let $X$ be a subset of $V'$.  
By Proposition \ref{efface}, it is enough to prove that either at least two arcs  or at least one arc of $\vec{E'}$  leave $X.$

If there are $x\in X$ and $y\in V'-X$  of degree at least 4, then, since $G'$ is essentially 4-edge-connected,  there is no 3-edge-cut separating $x$ and $y$ in $G'$ and therefore, as $\vec{F'}$ is well-balanced, there are 2 arcs leaving  $X$, and we are done. 

Hence, by considering $V'-X$ and $\cev {G'}$ if necessary, we may  assume without loss of generality that $X$ only contains vertices of degree 3 and there is no arc of $\vec{E'}$  leaving $X$. By construction, every vertex $v$ of $X$ has at least one arc $\vec{e}_v^1$ or $\vec{e}_v^2$ of  $\vec{E'}$ leaving $v$. As there is no arc of $\vec{E'}$ leaving $X$, we obtain that $\vec{G'}[X]$ contains a circuit $\vec{C}$ of arcs in $\vec{E'}$. This cycle $C$ provides a contradiction since $G'-M$ is a forest.
 \end{proof}
\medskip

We are now ready to prove a cubic version of Theorem \ref{34ec}.

\begin{theo}\label{cubice4ec}
Let $G$ be a cubic essentially 4-edge-connected graph. Then $f(G)\leq 3$.
\end{theo}
 
\begin{proof} 
Since $G$ is cubic and $2$-edge-connected, by Theorem \ref{pettheo}, $G$ has a perfect matching  $M_1$ and the connected components of $G-M_1$ form a packing ${\cal C}$ of cycles. 
%Let ${\cal C}_o$  be the set of cycles in ${\cal C}$ of odd  length.  
By Lemma \ref{matching}, there exists an orientation $D_1$ of $G$ such that each cycle of ${\cal C}$ is oriented as a circuit and $M_1$ is deletable in $D_1.$  By Lemma \ref{cyclearc}, each  $C_i\in {\cal C}$ contains a deletable  arc $\vec e_i$ in $D_1$. Note that the connected components of  $G-M_1-\cup \{e_i:{C_i\in {\cal C}}\}$ form a packing of  paths which is the union of two matchings $M_2$ and $M_3.$ By Lemma \ref{matching}, there exist orientations $D_2$ and $D_3$ of $G$ such that $M_2$ is deletable in $D_2$ and $M_3$ is deletable in $D_3.$ Since $E(G)=M_1\cup M_2\cup M_3\cup \{e_i:{C_i\in {\cal C}}\},$ Theorem \ref{cubice4ec} follows.
 \end{proof}
 %%%%%%%%%%%%%%%%%%%%%%%%%%%%%%%%%%%%%%%%%%%%%%%%%%%%%%
\subsection{Non-cubic case}
\medskip

We now generalize the results of the previous part to arbitrary essentially $4$-edge-connected graphs.
\medskip

\begin{proof} (of Theorem \ref{34ec}).
 Let  $G=(V,E)$ be a counterexample minimizing $|V|$.

\begin{claim}\label{cutvertexbridge}
$G-v$ is 2-edge-connected for all $v\in V.$
\end{claim}

\begin{proof}
For the sake of a contradiction, assume that $G-v$ is not 2-edge-connected for some $v\in V.$ If $G-v$ is disconnected, we obtain a contradiction using the same argument as in the proof of Claim \ref{cutvertex}. We therefore have a partition $A_1\cup A_2$ of $V-\{v\}$ such that $A_1$ and $A_2$ are only connected by a single edge $e_0$ in $G-v$. Let us denote the end-vertices of $e_0$ by $u_i\in A_i.$  Consider the graph $G_{i}$ that arises from $G$ by contracting $A_{3-i} \cup \{v\}$ into a vertex $v_i$. Note that $E(G_1) \cap E(G_2) =\{e_0\}$. Since $G$ is essentially 4-edge-connected, so is $G_i$. Moreover, $G_i$ is smaller than $G.$ It follows that  there exist 3 orientations $D^i_{j}$  of $G_i$ such that one of $D^i_{j}-\vec e$ is strongly connected for all $e\in E(G_i)$. We may suppose that $D^1_{1}-\vec e_0$ and $D^2_{1}-\vec e_0$ are strongly connected. Reversing the arcs in $D^i_{j}$ if needed, we may assume that $e_0$ has the same orientation in $D^1_{j}$ and $D^2_{j}$. We can construct the 3 orientations $D_j$ of $G$ by merging $D^1_{j}$ and $D^2_{j}$. We will finish the proof by showing that for all $e \in E,$ there exists a $j$ such that $D_j-\vec e$ is strongly connected. Let $e \in E$ and $j \in \{1,2,3\}$ such that both $D^1_{j}-\vec{e}$ and $D^2_{j}-\vec{e}$ are strongly connected.  Observe that if $e \neq e_0$, then either $D^1_{j}-\vec e=D^1_{j}$ or $D^2_{j}-\vec e=D^2_{j}$. Assume that there is a nonempty, proper subset $X$  of $V$ that has no arc leaving in $D_j-\vec e$. Without loss of generality, we may assume that $v \in X$. 
%If one of $u_1,u_2$, say $u_1$ is in $X$, then there is no arc leaving $(X\cap A_1)\cup \{v_1\}$ in $D_{1,j}-\vec e$, a contradiction. We may therefore assume that $u_1,u_2 \notin X$. 
As $(X\cap A_i)\cup \{v_i\}$ has an arc leaving in  $D^i_{j}-\vec  e$, $e_0$ must be directed away from $v_i$ in  $D^i_{j}$ for $i=1,2$.  This is a contradiction as $D^1_{j}$ and $D^2_{j}$ were chosen to  both have the same orientation of $e_0$.
\end{proof}
\bigskip

Let $H_G$ be a cubic extension of $G$ as defined in Section \ref{reduc}. By Claim \ref{cutvertexbridge} and Proposition \ref{red34}(b), $H_G$ is a cubic essentially 4-edge-connected graph. Then, by Theorem \ref{cubice4ec}, the Frank number of $H_G$ is at most $3$. There exist therefore $3$ orientations $D'_j$ of $H_G$ such that for each edge $e \in E(H_G)$, there is some $j \in \{1,2,3\}$ such that $D'_j-\vec e$ is strongly connected. Consider now the $3$ orientations $D_j$ of $G$ which arise from $D'_j$  by contracting the subgraphs $C_v$ for all $v\in V.$ By Proposition \ref{strnglcnnctd}(a), if $D'_j-\vec e$ is strongly connected for an edge $e\in E$ , so is $D_j-\vec e$. It follows that the Frank number of $G$ is at most $3$, a contradiction.
\end{proof}

%%%%%%%%%%%%%%%%%%%%%%%%%%%%%%%%%%%%%%%%%%%%%%%%%%%%%%
\section{The Petersen graph}\label{petorient}

In this section, we show that there are graphs of Frank number higher than two, more precisely we prove Theorem \ref{pet3}. While this result can also be established computationally, we prefer to give a proof by hand.
\medskip

\begin{proof} (of Theorem \ref{pet3})
Let {\boldmath $G$} $=(V,E)$ be the Petersen graph, see Figure \ref{fig0}. We frequently make use of the symmetry properties of $G$.  
By Theorem \ref{34ec} and since $G$ is essentially 4-edge-connected, but not 4-edge-connected, it suffices to prove that its Frank number is different from 2.  Suppose that $G$ has Frank number 2 and let {\boldmath $D_1$} $=(V,A_1)$ and {\boldmath $D_2$} $=(V,A_2)$ be  two orientations of $G$ such that

\begin{equation}\label{franknumber2}
\text{$D_1-\vec e$ or $D_2-\vec e$ is strongly connected for each edge $e$ of $G$.} \tag{$*$}
\end{equation}  

We say that an arc of $D_1$ is {\it stable} if the same arc exists in $D_2,$ otherwise it is {\it changing}. 
Let {\boldmath $S$} and {\boldmath $C$} be the set of stable and changing arcs, respectively.
Note that $D_1$ and $\cev{D_2}$ also satisfy \eqref{franknumber2} and stable and changing arcs are exchanged. Hence, whatever is proved for stable arcs is also true for changing arcs.
\medskip

We first show that $S$ and $C$ induce a $2$-edge-coloring of $G$ with certain properties and then that no such $2$-edge-coloring exists. Observe that none of the considered colorings are required to be proper.
For a  2-edge-coloring  $R,B$ of $G,$ we define  an auxiliary graph {\boldmath $H^{R,B}$} $:=(V,F)$  where $uv\in F$ if there exists a 3-path $tuvw$ in $G(R)$ or in $G(B)$ or there exists  a $(u,v)$-path that is a connected component of $G(R)$ or of $G(B)$.

 \begin{lemma}\label{main1}
	$G$ has a 2-edge-coloring $R,B$ such that 
 		\begin{equation}
 			\text{\rm  no monochromatic 3-star exists,}\label{star}
			\end{equation}
			\begin{equation}
 			\text{\rm  $H^{R,B}$ is bipartite.}\label{oddcycle}
 		\end{equation}
\end{lemma}

\begin{proof}
We show that the 2-edge-coloring induced by $S$ and $C$ satisfies (\ref{star}) and (\ref{oddcycle}).
To show (\ref{star}) we need the following claim.

\begin{claim}\label{covering}
Each vertex is incident to at least one stable arc.
\end{claim}

\begin{proof}
Suppose that a vertex $v$ is incident only to changing arcs. Since $G$ is cubic and $D_1$ is strongly connected, either the in-degree or the out-degree of $v$ is 1, say $\vec{e}$ is the only arc entering $v$. Then $\cev{e}$ is the only arc leaving $v$ in $D_2.$ Then, $D_1-\vec{e}$ and $D_2-\cev{e}$ are not strongly connected, which is a contradiction.
\end{proof}
 \medskip

To show (\ref{oddcycle}) we need the following claims.

\begin{claim}\label{dipath}
The weakly connected components of $D_1(S)$ are directed paths or circuits.
\end{claim}

\begin{proof}
By Claim \ref{covering} applied for stable arcs and then for changing arcs, the connected components of $D_1(S)$ are paths or cycles.
If two stable arcs are incident to a vertex $v$ then one of them enters and the other one leaves $v.$ Otherwise, let $e$ be the third arc incident to $v.$ Then, $D_1-\vec e$ and $D_2-\vec e$ are not strongly connected, which is a contradiction.
Now the claim follows.
\end{proof}
 
 \begin{claim}\label{end-vertices}
Let $P$ be a weakly connected component of $D_1(S)$ that is a directed $(u,v)$-path.
Then the in-degrees of  $u$ and $v$ in $D_1$ are of different parity.
\end{claim}

 \begin{proof}  Since $G$ is cubic and $u$ and $v$ are incident to exactly one stable arc in $D_1$, $u$ and $v$ are incident to exactly two changing arcs in $D_1$. Then, by Claim \ref{dipath} applied for $D_1(C)$, exactly one changing arc enters both $u$ and $v$ in $D_1$. Since $P$ is a directed path between $u$ and $v$, the claim follows.
 \end{proof}

\begin{claim}\label{miedge}
Let $tuvw$ be a 3-path in $D_1(S).$ Then the  in-degrees of $u$ and $v$ are of different parity in $D_1$.
\end{claim}

 \begin{proof} By Claim \ref{dipath}, exactly one stable arc enters both $u$ and $v$ in $D_1$. By Claim \ref{covering}, the two other arcs incident to $u$ and $v$ are changing. If both are entering or leaving then $D_1-uv$ and $D_2-uv$ are not strongly connected, which is a contradiction. Now the claim follows.
 \end{proof}

\begin{claim}\label{bipartite}
$H^{S,C}$ is a bipartite graph.
\end{claim}

\begin{proof}
Since $G$ is cubic and $D_1$ and $D_2$ are strongly connected, each vertex is of in-degree $1$ or $2.$ By Claims \ref{end-vertices} and \ref{miedge}, each edge of $H^{S,C}$ is between a vertex of in-degree 1 and a vertex of in-degree 2, so $H^{S,C}$ is bipartite.
\end{proof}
\medskip

By Claim \ref{covering} applied for $R:=S$ and $B:=C$ and by Claim \ref{bipartite}, Lemma \ref{main1} follows.
\end{proof}
\medskip

We show that $G$ does not admit any 2-edge-coloring satisfying  (\ref{star}) and (\ref{oddcycle}) and obtain a contradiction to Lemma \ref{main1}.
\medskip

The following result yields a strong property such a coloring would have to satisfy.

\begin{lemma}\label{propcol}
Let $R,B$ be a 2-edge-coloring satisfying  (\ref{star}) and (\ref{oddcycle}). Then $G$ has a 5-cycle that contains a monochromatic 4-path whose end-vertices are incident to 2 edges of the other color.
\end{lemma}

\begin{proof}
We first show two weaker statements which are useful in the proof later on.
\begin{claim}\label{3pathexists}
$G$ has a monochromatic 3-path.
\end{claim}

\begin{proof}
Suppose not. 
%\begin{equation}\label{3p}
%\text{\rm  no 5-cycle contains a monochromatic 3-path.}
%\end{equation}
Since $G$ is cubic, there are two adjacent edges of the same color, without loss of generality $ab, ae\in R$. Then, by the assumption for $deab,eabc, eabi$ and $heab$, we obtain that $de, bc, bi, eh\in B$. Thus,  by the assumption for $cbih, cbij$ and $dehg$, we obtain that $jihg$ forms a monochromatic 3-path, contradicting the assumption. See Figure \ref{fig1}.
\end{proof}
\medskip

\begin{figure}
    \centering
    \begin{subfigure}{0.25\textwidth}
        \includegraphics[width=\textwidth]{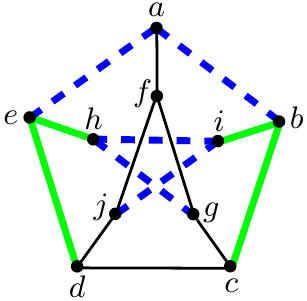}
        \caption{}\label{fig1}
    \end{subfigure}
    \hskip 1truecm
    \begin{subfigure}{0.25\textwidth}
        \includegraphics[width=\textwidth]{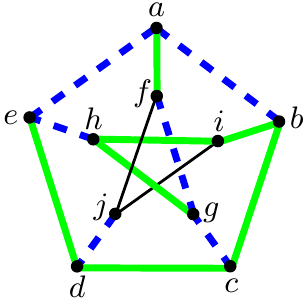}
        \caption{}\label{fig2}
    \end{subfigure}
    \hskip 1truecm
    \begin{subfigure}{0.25\textwidth}
        \includegraphics[width=\textwidth]{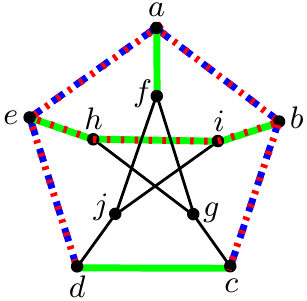}
        \caption{}\label{fig3}
    \end{subfigure}
    \medskip
    
    \begin{subfigure}{0.25\textwidth}
        \includegraphics[width=\textwidth]{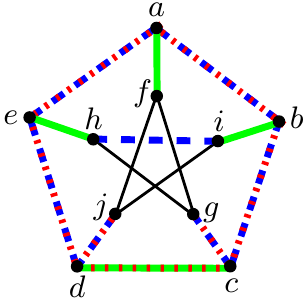}
        \caption{}\label{fig4}
    \end{subfigure}
    \hskip 1truecm
    \begin{subfigure}{0.25\textwidth}
        \includegraphics[width=\textwidth]{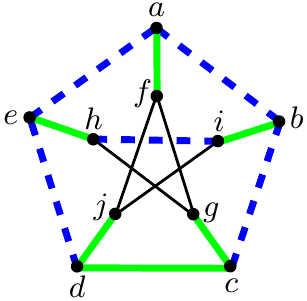}
        \caption{}\label{fig5}
    \end{subfigure}
     \hskip 1truecm
   \begin{subfigure}{0.25\textwidth}
        \includegraphics[width=\textwidth]{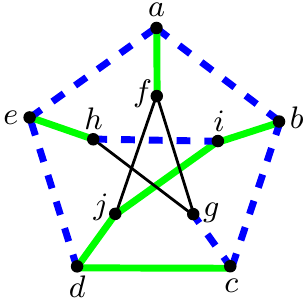}
        \caption{}\label{fig6}
    \end{subfigure}
    \medskip

    \begin{subfigure}{0.25\textwidth}
        \includegraphics[width=\textwidth]{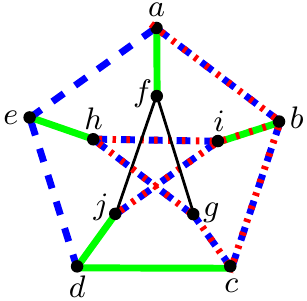}
        \caption{}\label{fig7}
    \end{subfigure}
    \hskip 1truecm
    \begin{subfigure}{0.25\textwidth}
        \includegraphics[width=\textwidth]{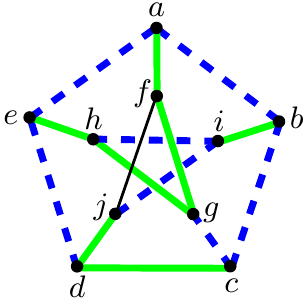}
        \caption{}\label{fig8}
    \end{subfigure}
    \hskip 1truecm
    \begin{subfigure}{0.25\textwidth}
        \includegraphics[width=\textwidth]{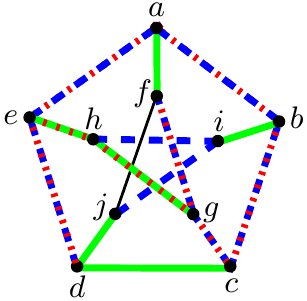}
        \caption{}\label{fig9}
    \end{subfigure}
     \medskip

   \begin{subfigure}{0.25\textwidth}
        \includegraphics[width=\textwidth]{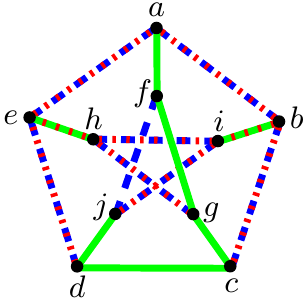}
        \caption{}\label{fig10}
    \end{subfigure}
    \hskip 1truecm
    \begin{subfigure}{0.25\textwidth}
        \includegraphics[width=\textwidth]{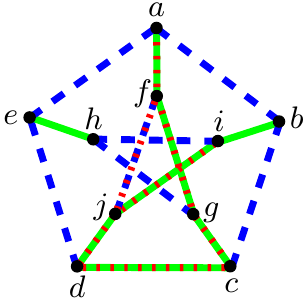}
        \caption{}\label{fig11}
    \end{subfigure}
    \hskip 1truecm
    \begin{subfigure}{0.25\textwidth}
        \includegraphics[width=\textwidth]{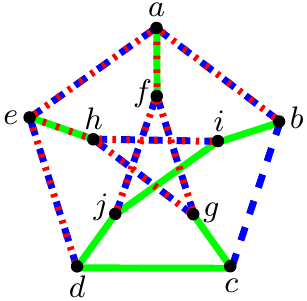}
        \caption{}\label{fig12}
    \end{subfigure}
    \medskip

    \begin{subfigure}{0.25\textwidth}
        \includegraphics[width=\textwidth]{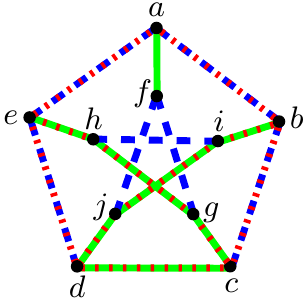}
        \caption{}\label{fig13}
    \end{subfigure}
    \caption{}\label{fig}
\end{figure}

\medskip

This result is helpful in proving a strengthening of itself.
\begin{claim}\label{4pathexists}
$G$ has a 5-cycle that contains a monochromatic 4-path.
%\begin{equation}\label{4p}
%\text{\rm  no 5-cycle contains a monochromatic 4-path.}
%\end{equation}
\end{claim}

\begin{proof}
Suppose not. 
By Claim \ref{3pathexists}, without loss of generality $bc, cd, de\in B$. Then, by the assumption for $abcde$, we obtain that $ab, ae\in R$. By (\ref{star}) for $a, c$ and $d$, we obtain that $af\in B$ and $cg, dj\in R$. By (\ref{star}) for $f$, one of $fg$ and $fj$ is in $R$. By symmetry, without loss of generality  $fg\in R$. Then,  by (\ref{star}) for $g$, $gh\in B$. So,  by the assumption for $cdehg,$ we obtain that $eh\in R$. Then,  by the assumption for $abihe$, we obtain that $hi, bi\in B$. Thus $cbihg$ forms a monochromatic 4-path in the 5-cycle $cbihg$, that contradicts the assumption. See Figure \ref{fig2}.
 \end{proof}
 \medskip

By Claim \ref{4pathexists}, without loss of generality $ab, bc, de, ea\in R$. Then, by (\ref{oddcycle}) for $abcde$, we obtain that $cd\in B$.
By (\ref{star}) for $a, b$ and $e$, we obtain that $af, bi, eh\in B$. %See Figure 3.
If $hi\in B$, then, by the 3-paths of  $deabc$ and by  $ehib$, we obtain that $H^{R,B}$ contains the 3-cycle $abe$ that contradicts (\ref{oddcycle}). 
See Figure \ref{fig3}.
Hence, $hi\in R$. %See Figure 5.
If  $cg, dj\in R$, then, by the 3-paths of  $jdeabcg$ and by  $cd$, we obtain that $H^{R,B}$ contains the 5-cycle $abcde$ that contradicts (\ref{oddcycle}).
See Figure \ref{fig4}.
Hence, by symmetry, we may suppose that $dj\in B$. %See Figure 7.
\smallskip

Now suppose for the sake of a contradiction that $G$ does not contain a 5-cycle that contains a monochromatic 4-path whose end-vertices are incident to 2 edges of the other color.
 If $cg\in B$, then  $abcde$ contradicts the assumption. See Figure \ref{fig5}. Hence $cg\in R$. %See Figure 9.
If $ij\in B$, then  $bcdji$ contradicts the assumption. See Figure \ref{fig6}. Hence $ij\in R$. %See Figure 11.
If $gh\in R$, then, by the 3-paths of  $abcghij$ and by  $bi$, we obtain that $H^{R,B}$ contains the 5-cycle $bcghi$ that contradicts (\ref{oddcycle}). See Figure \ref{fig7}.
Hence $gh\in B$. %See Figure 13.
If $fg\in B$, then  $afghe$ contradicts the assumption. See Figure \ref{fig8}. Hence $fg\in R$. %See Figure 15.
Then, by the 3-paths of  $deabcgf$ and by  $ehg$, we obtain that $H^{R,B}$ contains the 5-cycle $abcge$ that contradicts (\ref{oddcycle}).  
See Figure \ref{fig9}. This finishes the proof of Lemma \ref{propcol}.
 \end{proof}
 \medskip

Lemma \ref{propcol} yields that $G$ has a 5-cycle, without loss of generality $abcde$,  that contains a monochromatic 4-path whose end-vertices are incident to 2 edges of the other color.  By similar arguments as before, we obtain the partial coloring of Figure \ref{fig5}. By (\ref{star}) for $f$, one of $fg$ and $fj$ is in $R$. By symmetry, without loss of generality $fj\in R$. %See Figure 18.

	 Suppose that $fg\in B$.  Then, by (\ref{star}) for $g$, $gh\in R$. 
 If $ij\in R$, then, by the 3-paths of  $deabc$ and for $ghij$, and by  $eh$ and $ib$, $H^{R,B}$ contains the 5-cycle $eabih$ contradicting (\ref{oddcycle}). See Figure \ref{fig10}.
 If $ij\in B$, then, by the 3-paths of  $afgcdji$ and by  $fj$, we obtain that $H^{R,B}$ contains the 5-cycle $fgcdj$ contradicting (\ref{oddcycle}). See Figure \ref{fig11}.

	  Hence $fg\in R$. Then, by (\ref{oddcycle}) for $fghij$, one of $hg$ and $ij$ is in $B$. By symmetry, we may suppose that $ij\in B$. 
 If $hg\in R$, then, by the 3-paths of  $jfghi$ and for $deab$, and by  $eh$ and $af$, we obtain that $H^{R,B}$ contains the 5-cycle $fghea$ contradicting (\ref{oddcycle}). See Figure \ref{fig12}.
 If $hg\in B$, then, by the 3-paths of  $deabc$ and by  $bijdcghe$, we obtain that $H^{R,B}$ contains the 3-cycle $eab$ contradicting (\ref{oddcycle}). See Figure \ref{fig13}. 
\medskip

In all cases we obtain a contradiction which implies that $G$ has Frank number different from $2$. This finishes the proof of Theorem \ref{pet3}.
\end{proof}

\section{Algorithmic aspects}\label{algo}

This section is dedicated to proving Theorem \ref{nphard}.

Our reduction is from a slightly stronger variation of MNAE3SAT. In the first part, we introduce this problem and show that it is NP-complete by a reduction from MNAE3SAT. Next, we introduce our construction and show that the constructed graph is cubic and 3-edge-connected. The last two parts are dedicated to showing that the reduction works indeed.

\subsection{Boolean formulas}
\medskip

Given a MNAE3SAT formula $F=(X,\mathcal{C})$, we call a truth assignment to the variables of $X$ {\it feasible} if every clause of $\mathcal{C}$ contains at least one true and  at least one false literal. We define the formula graph {\boldmath $G_F$} by $V(G_F)=X \cup \mathcal{C}$ and there is an edge between the vertices corresponding to a variable $x_i$ and a clause $C_j$ if $x_i$ is contained in $C_j$. We call a formula $F$ {\it connected} if $G_F$ is connected. We show that MNAE3SAT stays NP-complete with this additional assumption.

\bigskip
\noindent {\bf Connected Monotone Not-all-equal-3SAT(CMNAE3SAT)}

\smallskip
\noindent Instance: A set $X$ of boolean variables, a connected formula consisting of a set $\mathcal{C}$ of clauses each containing 3 distinct variables none of which are negated.

\smallskip
\noindent Question: Is there a feasible truth assignment to the variables of $X$?

\begin{lemma}
CMNAE3SAT is $NP$-complete.
\end{lemma}
\begin{proof}
We show a reduction from MNAE3SAT. Recall that MNAE3SAT is NP-complete by Theorem \ref{sathard}. Let $F$ be a MNAE3SAT formula. Let $G_1,\ldots,G_t$ be the connected components of $G_F$. For $i=1,\ldots,t$, consider the MNAE3SAT formula $F_i$ that consists of the variables and clauses corresponding to vertices in $G_i$. Observe that $G_{F_i}=G_i$ and so every $F_i$ is an instance of CMNAE3SAT. We will show that $F$ is a positive instance of MNAE3SAT if and only if all of the $F_i$ are positive instances of CMNAE3SAT. First assume that there is a feasible truth assignment for $F$. The restriction of this assignment to the variables of $F_i$ yields a feasible truth assignment for $F_i$ for all $i=1,\ldots,t$. Now assume that there is a feasible truth assignment for $F_i$ for $i=1,\ldots,t$. As every vertex corresponding to a variable is contained in exactly one component, every variable is contained in exactly one of the $F_i$ and so we obtain a unique assignment of boolean values to all variables. As every clause of $\mathcal{C}$ is contained in some $F_i$, this assignment is feasible for $F$. This finishes the proof.
\end{proof}

\subsection{The construction}
\medskip

Let $F=(X,\mathcal{C})$ be a CMNAE3SAT formula with $X=\{x_1,\ldots,x_m\}$. If there is a variable $x \in X$ that is contained in only one clause $C \in \mathcal{C}$, then $F$ is satisfiable if and only if $(X-\{x\},\mathcal{C}-\{C\})$ is satisfiable. We may therefore assume that every $x_i \in X$ is contained in at least 2 clauses. For $i=1,\ldots,m$, we define {\boldmath $p_i$} to be the number of clauses $x_i$ is contained in. 

We now construct an instance $(G=(V,E),S)$ of DELETABILITY. For $i=1,\ldots,m$, $G$ contains a cycle {\boldmath $K_i$} of length $2p_i$. We abbreviate $V(K_i)$ to {\boldmath $V_i$} and $E(K_i)$ to {\boldmath $E_i$}. Observe that $V_i$ can be partitioned into two stable sets in a unique way. We call one of these sets {\boldmath $A_i$} and the other one {\boldmath $B_i$}. Note that $|A_i|=|B_i|=p_i.$ For every clause $C$, $G$ contains a vertex {\boldmath $v_C$}. We denote $\{v_C:C \in \mathcal{C}\}$ by {\boldmath $V_{\mathcal{C}}$}. Further, $G$ contains a cycle {\boldmath $K$} of length $3|\mathcal{C}|$. We abbreviate $V(K)$ to {\boldmath $V_K$} and $E(K)$ to {\boldmath $E_K$}. 
We add a perfect matching between $\{v_C:x_i \in C\}$ and $A_i$ for every $i=1,\dots,m$ and  between  $\bigcup_{i=1}^m B_i$ and $V_K$. Observe that this is possible because $|A_i|=p_i$ and $|\bigcup_{i=1}^m B_i|=\sum_{i=1}^m p_i=3|\mathcal{C}|=|V_K|$. 
%For every $i=1,\dots,m$, we add a perfect matching between $\{v_C:x_i \in C\}$ and $A_i$. Observe that this is possible because there are exactly $p_i$ vertices in $A_i$. Next, we add a perfect matching between  $\bigcup_{i=1}^m B_i$ and $V_K$. Observe that this is possible because $|\bigcup_{i=1}^m B_i|=\sum_{i=1}^m p_i=3|\mathcal{C}|=|V_K|$. 
Finally, we define {\boldmath $S$} $= \bigcup_{i =1}^m E_i$.
Note that $|V|=10|\mathcal{C}|$ and $|E|=15|\mathcal{C}|$, so the construction is polynomial indeed.
\medskip

\begin{figure}[h]
        \centerline{\includegraphics[width=.6\textwidth]{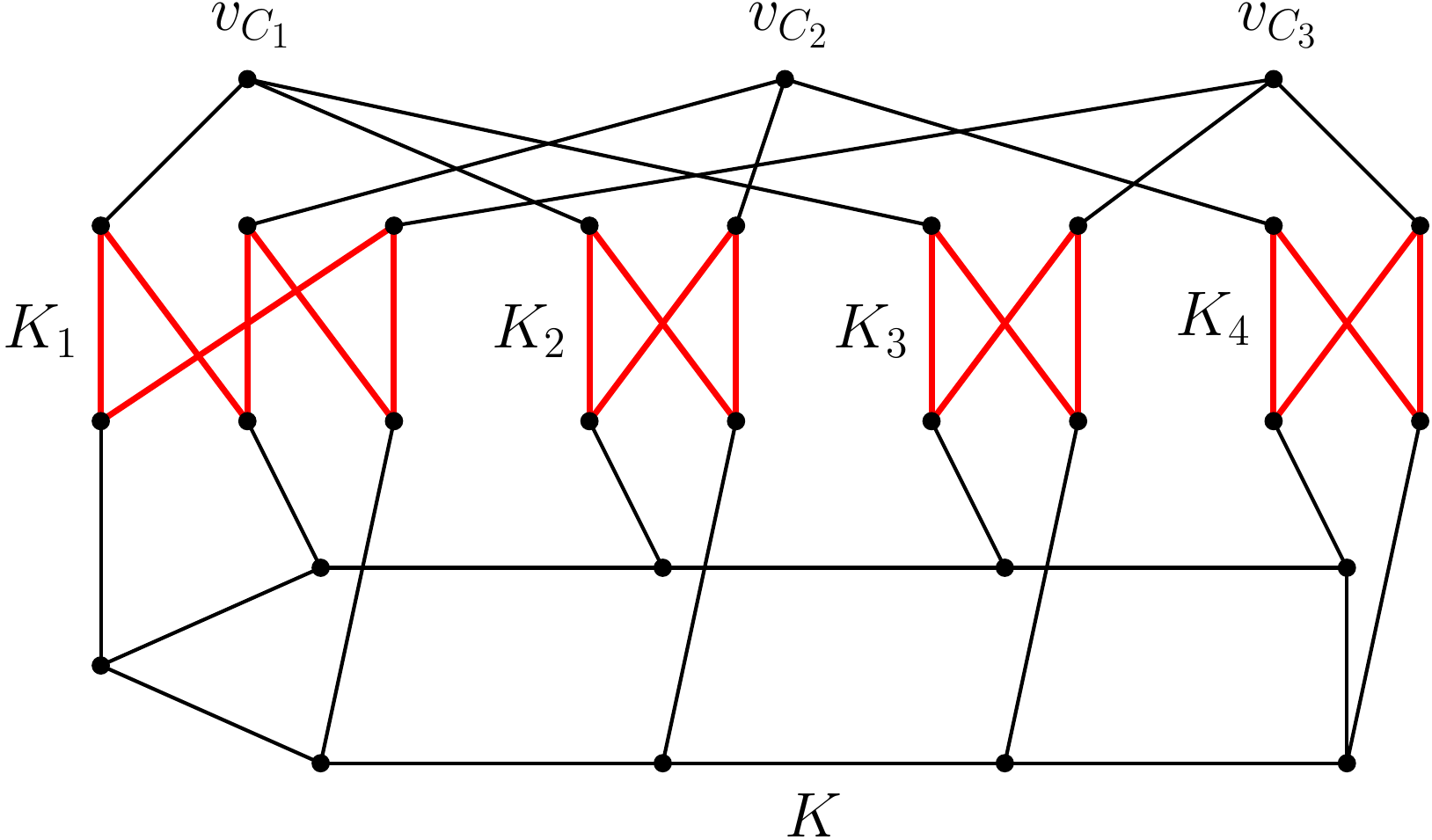}}
        \caption{}\label{constr}
\end{figure}

Figure \ref{constr} shows the constructed graph for the formula consisting of the variables $x_1,\ldots,x_4$ and the clauses $C_1=\{x_1,x_2,x_3\},C_2=\{x_1,x_2,x_4\}$ and $C_3=\{x_1,x_3,x_4\}$. The edges of $S$ are marked in red.
\medskip

 Observe that $G$ is cubic as every clause contains exactly 3 variables and by construction. We show that it also satisfies the other desired structural property.

\begin{lemma}\label{3-arete}
$G$ is 3-edge-connected.
\end{lemma}

\begin{proof}
Assume for the sake of a contradiction that $G$ contains some cut $\delta(Z)$ which consists of at most 2 edges. Without loss of generality, we may assume that $V_K \cap Z$ is nonempty. 

\begin{claim}\label{kz}
$V_K \subseteq Z$.
\end{claim}

\begin{proof}
Assume that there is a vertex $w \in V_K-Z$. As $G-V_K$ arises from $G_F$ by replacing vertices by cycles and $G_F$ is connected by assumption, $G-V_K$ is connected. Then, since a perfect matching exists between  $\bigcup_{i=1}^m B_i$ and $V_K$, we obtain that $G-E_K$ is also connected. As $K$ is 2-edge-connected, it follows that $2\geq d_G(Z)=d_K(Z)+d_{G-E_K}(Z)\geq2+1=3$, a contradiction. %There are two canonical edge-disjoint paths $P_1,P_2$ between $v$ and $w$ in $K$. Also, by construction, $v$ and $w$ have neighbors $v'$ and $w'$ in $V(G)-V(K)$. As $G-V(K)$ is connected, there is a path between $v'$ and $w'$ in $G-V(K)$. Together with the edges $vv'$ and $ww'$, this yields a third path between $v$ and $w$ which is disjoint from $P_1$ and $P_2$. So $v$ and $w$ are connected by 3 edge-disjoint paths, a contradiction to them being separated by a cut of at most 2 edges.
\end{proof}

\begin{claim}\label{cz}
$V_{\mathcal{C}} \subseteq Z$.
\end{claim}

\begin{proof}
Consider a vertex $v_C$ where $C$ contains the variables $x_i,x_j,x_{\ell}$. By construction, both $v_C$ and $K$ have a neighbor in each of $V_{i},V_{j}$ and $V_{\ell}$ and  $K_{i},K_{j}$ and $K_{\ell}$ are connected. As $V_K \subseteq Z$ by Claim \ref{kz}, there are 3 edge-disjoint paths from $v_C$ to $Z$. It follows, by $d_G(Z)\le 2,$ that $v_C \in Z$.
\end{proof}
\medskip

By Claims \ref{kz} and \ref{cz}, there exists a vertex $v \in V_i-Z$ for some $i=1,\ldots,m$ and $v$ is connected to $V_K\cup V_{\mathcal C}\subseteq Z$ by a path of length 1 and two paths of length 2 and all of these are edge-disjoint. % As every variable is contained in at least two clauses, $v_{x}^{j-1}$ and $v_{x}^{j+1}$ are distinct. Also, by construction, $v_{x}^{j-1+k}$ has a neighbor $w_k \in Z$ for $k=0,1,2$. This yields the 3 edge-disjoint paths $P_0=v_x^j v_{x}^{j-1}w_0, P_1=v_x^j w_1$ and $P_2=v_x^j v_{x}^{j+1}w_2 $ from $v_x^j $ to $Z$.
  This is a contradiction to $Z$ being separated from $v$ by a cut of at most 2 edges. This finishes the proof of Lemma \ref{3-arete}.
\end{proof}
\medskip

The remaining part of this section is dedicated to showing that our construction is indeed correct, i.e. $F$ is a positive instance of CMNAE3SAT if and only if $(G,S)$ is a positive instance of DELETABILITY.

\subsection{From orientation to truth assignment}
\medskip

Suppose that $(G,S)$ is a positive instance of DELETABILITY, so there is an orientation $D$ of $G$ such that $D-\vec{s}$ is strongly connected for all $s \in S$. Before finding a feasible truth assignment of the formula, we need the following result about the orientation.

\begin{claim}\label{allsame}
	Let $i \in \{1,\ldots,m\}$. Then all the arcs between $A_i$ and $B_i$ are directed in the same way.
\end{claim}

\begin{proof}
Let $v$ be any vertex of $K_i$ and $e,f$ the two edges of $K_i$ incident to $v.$ Since $e,f\in S,$ $D-e$ and $D-f$ are strongly connected. Then, as $G$ is cubic,  both of $e$ and $f$ are either entering or leaving $v$. Since $K_i$ is connected, the claim follows.
\end{proof}
\medskip

Using Claim \ref{allsame}, we now define a truth assignment of $X$ in the following way: a variable $x_i$ is assigned the value true if the arcs between $A_i$ and $B_i$ are directed from $B_i$ to $A_i$ and false if  the arcs between $A_i$ and $B_i$ are directed from $A_i$ to $B_i$. 

Consider a clause $C=\{x_i,x_j,x_{\ell}\}$. The vertex $v_C$ has one neighbor in each of $A_i,A_j$ and $A_{\ell}$ in $G$. As $D$ is strongly connected and $G$ is cubic, $v_C$ has one in-neighbor $w$, say in $A_{\ell}$ and $w$ has an in-neighbor in $D[V_{\ell}]$. It follows by construction that $x_{\ell}$ is set to true in the truth assignment. Similarly, one of $x_i,x_j,x_{\ell}$ is set to false. It follows that the assignment is feasible.
%For the sake of a contradiction, suppose that there is a clause $C=\{x_1,x_2,x_3\}$ which is not satisfied by this assignment. By construction, $v_C$ has exactly one neighbour in each $V(K_{x_i})$ for $i=1,2,3$. If all of $x_1,x_2,x_3$ are assigned to be true, this means that all 3 edges incident to $v_C$ are oriented towards $v_C$. If all of $x_1,x_2,x_3$ are assigned to be false, this means that all 3 edges incident to $v_C$ are oriented away from $v_C$. In either case, this is a contradiction to $D$ being strongly connected.

\subsection{From truth assignment to orientation}
\medskip

Assume that there is a feasible truth assignment for an instance $F$ of CMNAE3SAT consisting of a variable set $X=\{x_1,\ldots,x_m\}$ and a clause set $\mathcal{C}$. Relabeling variables, we may assume that there is some {\boldmath $t$} $\in \{0,\ldots,m\}$ such that $x_i$ is set to true for $i=1,\ldots,t$ and $x_i$ is set to false for $i=t+1,\ldots,m$. 
Let {\boldmath $\mathcal{A}_1$} $=\bigcup_{i=1}^t A_{i},$ {\boldmath $\mathcal{A}_2$} $=\bigcup_{i=t+1}^m A_{i},$ {\boldmath $\mathcal{B}_1$} $=\bigcup_{i=1}^t B_{i}$ and {\boldmath $\mathcal{B}_2$} $=\bigcup_{i=t+1}^m B_{i}.$ 

We define an orientation {\boldmath $D$} of $G$ as follows. We orient all edges  from $P$ to $R$ where $P$ and $R$ are two consecutive sets in $\mathcal{A}_1,V_{\mathcal{C}}, \mathcal{A}_2,\mathcal{B}_2, V_K,\mathcal{B}_1,\mathcal{A}_1.$
%$(P,R)\in \{(\mathcal{A}_1,V_{\mathcal{C}}), (V_{\mathcal{C}}, \mathcal{A}_2), (\mathcal{A}_2,\mathcal{B}_2), (\mathcal{B}_2, V_K), (V_K,\mathcal{B}_1), (\mathcal{B}_1,\mathcal{A}_1)\}.$
%We orient 
%all edges between $\mathcal{A}_1$ and $V_{\mathcal{C}}$ from $\mathcal{A}_1$ to $V_{\mathcal{C}}$,  
%all edges between $V_{\mathcal{C}}$ and $\mathcal{A}_2$ from $V_\mathcal{C}$ to $\mathcal{A}_2$, 
%all edges between $\mathcal{A}_2$ and  $\mathcal{B}_2$ from $\mathcal{A}_2$ to $\mathcal{B}_2$,  
%all edges between $\mathcal{B}_2$ and  $V_K$ from $\mathcal{B}_2$ to $V_K$,  
%all edges between $V_K$ and  $\mathcal{A}_1$ from $V_K$ to $\mathcal{A}_1$ and  
%all edges between $\mathcal{B}_1$ and  $\mathcal{A}_1$ from $\mathcal{B}_1$ to $\mathcal{A}_1$. 
Finally, we orient the edges of $K$ as a circuit. 
%We denote the obtained orientation by $D$.
\medskip

 \begin{figure}[h]
        \centerline{\includegraphics[width=.6\textwidth]{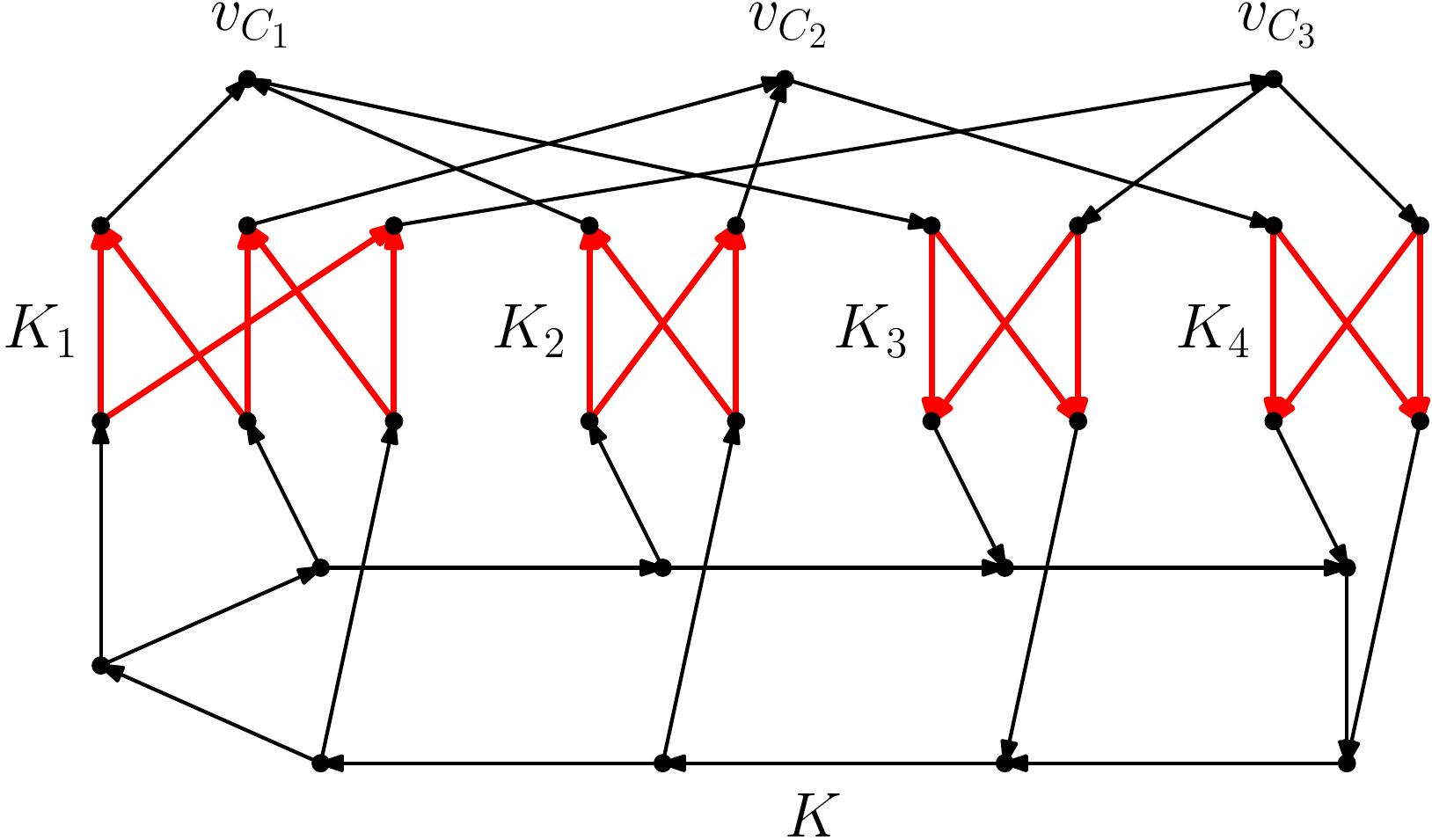}}
        \caption{}
\label{directed}
    \end{figure}
    
Figure \ref{directed} shows the obtained orientation for the formula consisting of the variables $x_1,\ldots,x_4$ and the clauses $C_1=\{x_1,x_2,x_3\},C_2=\{x_1,x_2,x_4\}$ and $C_3=\{x_1,x_3,x_4\}$ when $x_1$ and $x_2$ are set to true and $x_3$ and $x_4$ are set to false.
\medskip

The following is the orientation's decisive property:

\begin{claim}\label{inout}
In $D$, every vertex $v_C\in V_\mathcal{C}$ has an in-neighbor in $\mathcal{A}_1$ and an out-neighbor in $\mathcal{A}_2$. 
\end{claim}

\begin{proof}
Let $C$ contain the 3 variables $x_i,x_j$ and $x_{\ell}$. As the truth assignment is feasible, one of $x_i,x_j,x_{\ell}$, say $x_i$, is set to true and a different one, say $x_j$, is set to false. Then, by construction, $D$ contains an arc from $A_i\subseteq \mathcal{A}_1$ to $v_C$ and an arc from $v_C$ to $A_j\subseteq \mathcal{A}_2$. %For a contradiction and by symmetry, suppose that all three arcs entering $v_C$ are directed towards $v_C$. For $i=1,2,3$, there is some vertex $v_{x_i}^{j_i}$ with $j_i$ even such that $D$ contains an arc from $v_{x_i}^{j_i}$ to $v_C$. It follows that the arcs of $\vec{E(K_{x_i}}$ enter $v_{x_i}^{j_i}$. It follows by construction that $x_i$ is assigned the value true in the assignment. This is a contradiction as the clause $C$ is not satisfied then.
\end{proof}
\medskip

The following result will finish the proof:

\begin{claim}
Let $s \in S$. Then $D-\vec{s}$ is strongly connected.
\end{claim}

\begin{proof} Since $K$ is oriented as a circuit, all vertices of $K$ are in the same strongly connected component {\boldmath $Q$}. 
%\medskip
By construction, all vertices in $\mathcal{B}_1$ have an in-neighbor in $V_K\subseteq Q$ and all vertices in $\mathcal{A}_1$ have 2 in-neighbors in $\mathcal{B}_1$ in $D$, so at least one in $D-\vec{s}$. It follows, by Claim \ref{inout}, that all vertices in $\mathcal{A}_1 \cup \mathcal{B}_1\cup V_\mathcal{C}$ are reachable from $Q$.
%First consider a vertex $v \in V_{x_i}$ for some $i=1,\ldots,t$. If $v \in B_{x_i}$, both arcs of $D[V{x_i}]$ are directed away from $v$. It follows that there is an arc of $\vec{E}(G)-\vec{S}$ entering $v$ whose head is in $Q$ because $v$ has at least one in-neighbor by construction. If $v \in A_{x_i}$, it has an in-neighbor in $B_i$. It follows that all vertices in $\bigcup_{i=1}^t V{x_i}$ are reachable from $Q$.
 By similar arguments, $Q$ is reachable from all vertices in $\mathcal{A}_2 \cup \mathcal{B}_2\cup V_\mathcal{C}$. 
%\medskip
%Now consider a vertex $v_C\in V_\mathcal{C}$.  Since all vertices in $\mathcal{A}_1$ are reachable from $Q$, by Claim \ref{inout}, so is $v_C$. Since $Q$ is reachable from all vertices in $\mathcal{A}_2$, by Claim \ref{inout}, so is from $v_C$. 
This yields that $V_\mathcal{C}\subseteq Q$.
%\medskip
Finally, from every vertex in $\mathcal{A}_1 \cup \mathcal{B}_1$ there exists a directed path of length 1 or 2 to a vertex $v_C\in V_\mathcal{C}$. Similarly, to every vertex in $\mathcal{A}_2 \cup \mathcal{B}_2$ there exists  a directed path of length 1 or 2 from a vertex  $v_C\in V_\mathcal{C}$. It follows that $D-\vec{s}$ is strongly connected.
\end{proof}
\bigskip

This reduction proves Theorem \ref{nphard}.

\section{Conclusion}\label{conc}

Our work shows that $f(G)\leq 7$ for every $3$-edge-connected graph $G$ and that $f(G)=3$ if $G$ is the Petersen graph. Also, we show a better bound for the more restricted classes of essentially $4$-edge-connected graphs and $3$-edge-colorable, $3$-edge-connected graphs. Further, we show that a graph of Frank number bigger than $5$ would imply the failure of Conjecture \ref{bf}. Moreover, the decision problem  whether all edges of a given subset can become deletable in one orientation is proven to be NP-complete.
\smallskip

The most obvious remaining problem is to improve these bounds on the Frank number in the general case. Considering the indications found during our work, we propose the following conjecture:

\begin{conj}\label{3gen}
Every $3$-edge-connected graph $G$ satisfies $f(G)\leq 3$.
\end{conj}

A possible way to make progress towards Conjecture \ref{3gen} would be the following generalization of Lemmas \ref{perfect} and \ref{matching}. Using the fact that cubic graphs are $4$-edge-colorable \cite{vizing} and similar arguments as before, Conjecture \ref{couplagegen} would imply that $f(G)\leq 4$ for any $3$-edge-connected graph.

\begin{conj}\label{couplagegen}
Let $M$ be a matching of a $3$-edge-connected graph $G$ intersecting each $3$-edge-cut of $G$ in at most one edge. Then $M$ is deletable.
\end{conj}

It would also be interesting to generalize Frank numbers to arbitrary odd connectivity:

\begin{opprob}
Given a $(2k+1)$-edge-connected graph $G$, what is the minimum number of  $k$-arc-connected orientations such that each edge becomes an arc whose deletion does not destroy $k$-arc-connectivity in at least one of these orientations?
\end{opprob}

It follows from a theorem in \cite{djs} that this number is bounded by a constant depending only upon $k$. We are particularly interested in whether or not this number can be bounded by a constant not depending upon $k$.

\section{Acknowledgements}\label{ack}
We would like to express our gratitude to Andr\'as Frank who introduced us to the problem. We also wish to thank Alantha Newman who made us aware of the results of  \cite {djs}.

\end{document}